\documentclass[12pt]{article}
\usepackage{amssymb,amsfonts, amsmath,amsthm,epsfig,tabularx}
\usepackage{caption}

\usepackage{hyperref}

\newtheorem{theorem}{Theorem}

\newtheorem{corollary}[theorem]{Corollary}

\newtheorem{conjecture}[theorem]{Conjecture}

\newtheorem{problem}[theorem]{Problem}

\newcommand{\IR}{\sigma_t^{f(n)}}

\newcommand{\irr}{{\rm irr}}
\newcommand{\imb}{{\rm imb}}

\newcommand{\beq}{\begin{eqnarray}}
\newcommand{\eeq}{\end{eqnarray}}

\newcommand{\beqs}{\begin{eqnarray*}}
\newcommand{\eeqs}{\end{eqnarray*}}

\begin{document}

\title{Extremizing antiregular graphs by  modifying total $\sigma$-irregularity}
\author{ Martin Knor$^{a}$,
Riste \v{S}krekovski$^{b,c,d,e}$, Slobodan Filipovski$^{c}$, Darko Dimitrov$^{b}$ \\[0.3cm] 
{\small $^{a}$ \textit{Slovak University of Technology in Bratislava, Slovakia}} \\ [0.1cm] 
{\small $^{b}$ \textit{Faculty of Information Studies in Novo Mesto, Slovenia }}\\ [0.1cm] 
{\small $^{c}$ \textit{FAMNIT, University of Primorska, Koper, Slovenia }}\\[0.1cm]  
{\small $^{d}$ \textit{Faculty of Mathematics and Physics, University of Ljubljana, Slovenia }}\\[0.1cm] 
{\small $^{e}$ \textit{Rudolfovo -- Science and Technology Centre Novo Mesto, Slovenia }}\\[0.1cm] 
}

\maketitle

\begin{abstract} 
The total $\sigma$-irregularity is given by
$
\sigma_t(G) = \sum_{\{u,v\} \subseteq V(G)} \left(d_G(u) - d_G(v)\right)^2,
$
where $d_G(z)$ indicates the degree of a vertex $z$ within the graph $G$.
It is known that the graphs maximizing $\sigma_{t}$-irregularity are split graphs with only a few distinct degrees.
Since one might typically expect that graphs with as many distinct 
degrees as possible achieve maximum irregularity measures, we modify this invariant to
$
\IR(G)= \sum_{\{u,v\} \subseteq V(G)} |d_G(u)-d_G(v)|^{f(n)},
$
where $n=|V(G)|$ and $f(n)>0$. We study under what conditions the above modification obtains its maximum 
for antiregular graphs. We consider general graphs, trees, and chemical graphs, and accompany our results 
with a few problems and conjectures.

\end{abstract}

\section{Introduction and preliminaries}

We restrict our study to undirected graphs with a finite number of vertices, excluding any graphs that have loops or parallel edges.
For any terminology or notation not explicitly defined herein, we direct the reader 
to the comprehensive textbook by Bondy and Murty \cite{bm-gt-2008}.

The degree of a vertex $v$ in a graph $G$,   $d_G(v)$, is defined as the number of edges incident to that vertex. 
A graph \( G \) is labeled \emph{regular} when every vertex has the identical degree; conversely, 
it is labeled \emph{irregular}. An invariant of the graph \( G \), symbolized by \( I(G) \), is referred 
to as an irregularity metric or index of irregularity if it meets the conditions 
\( I(G) \geq 0 \) and \( I(G) = 0 \) exclusively when the graph \( G \) is regular.
In this study, we delve into irregularity measures of graphs, focusing on differences 
between pairs of vertices within a graph.

Let $G$ have $n$ vertices. For any vertex $v \in V(G)$, its degree $d_G(v)$ satisfies
$0 \leq d_G(v) \leq n-1$. If $d_G(v) = n-1$, it implies that all other vertices have positive degrees,
ensuring that no vertex is isolated. Consequently, there cannot exist a graph where all vertices have distinct degrees.
The most favorable scenario in this context is to have $n-1$ distinct degrees, with one degree repeated.
Graphs with such configurations are referred to by various names, with the most widely recognized term being {\it antiregular.}
Notably, there exist precisely two antiregular graphs on $n$ vertices, with one being the complement of the other.
As one of these graphs exhibits vertices with degrees ranging from $0$ to $n-2$, it naturally consists of disconnected components.
Consequently, there exists only one connected antiregular graph on $n$ vertices.
This specific graph features vertices with degrees ranging from $1$ to $n-1$, with the degree $\lfloor\frac{n}{2}\rfloor$
occurring twice \cite{bc-ngp-1967}. For further exploration of properties and findings related to
antiregular graphs, interested readers are directed to the survey by Ali \cite{a-sag-2020}.

The \emph{imbalance} of an edge $e=uv \in E$ is defined as ${\imb}(e)=\left|d_G(u)-d_G(v)\right|$.
In \cite{AlbertsonIrr}, Albertson defined the \emph{irregularity} of $G$ as
sum  of imbalances of all edges of a graph, i.e.,
\beq \label{eqn:003}
{\irr}(G) = \sum_{e \in E(G)}  {\imb}(e) = \sum_{uv \in E(G)} | d_G(u) - d_G(v) |. 
\eeq
Another irregularity measure introduced in \cite{Dimit-Abdo} is closely connected to (\ref{eqn:003}). 
Similar to (\ref{eqn:003}), this measure also quantifies the irregularity of a graph solely based 
on the  differences in the degrees of its vertices. For a given graph \( G \), it is expressed as:
\beq\label{eqn:003-t}
\irr_t(G) =  \sum_{\{u,v\} \subseteq V(G)} \left|d_G(u)-d_G(v)\right|. \label{eqn:003-t}
\eeq

Due to its evident relationship with the irregularity measure $\irr(G)$, the invariant ${\rm irr}_t(G)$ 
is commonly referred to as the total irregularity of a graph. 
The total irregularity of a graph is solely determined by its degree sequence, making it an effective measure even when vertex adjacency details are unavailable. Notably, there are graphs exhibiting high $\text{irr}$ despite having very limited degree diversity, a characteristic unexpected in highly irregular graphs.
Conversely, as demonstrated in \cite{Dimit-Abdo}, graphs with maximal ${\rm irr_t}$ exhibit large degree sets, 
with some even possessing the largest possible ones. A comparison between irregularity 
and total irregularity was presented in \cite{Dar-Riste}.

 An alternative to the Albertson irregularity index, aiming to avoid the absolute value calculation, 
 led to the introduction of the irregularity index ${\sigma}(G)$ in \cite{gtycc-ipsi}. It is defined as follows:
\beq
{\sigma}(G) &=& \sum_{uv\in E(G)}(d_G(u) - d_G(v))^{2}.   \nonumber 
\eeq
Graphs with maximal $\sigma$-irregularity have been characterized in \cite{adg-gmsi-2018}, 
where lower bounds on $\sigma$-irregularity were also established. The inverse problem, 
which involves determining the existence of a graph with $\sigma$-irregularity equal to a 
given non-negative integer, was addressed in \cite{gtycc-ipsi, adg-gmsi-2018}. 
R{\' eti} \cite{r-spgiiprsi-2019} further explored $\sigma$-irregularity in comparison 
with various well-known irregularity measures across certain classes of graphs.

A connected graph is termed {\it \( k \)-cyclic} if it consists of \( n \) vertices and \( n + k - 1 \) edges.
 In \cite{aaabh-msikcgi-2023}, the study determined connected $k$-cyclic graphs exhibiting maximal $\sigma$-irregularity.

If a sequence $\mathcal{D} = (d_1, d_2, \ldots, d_n)$ corresponds to the degrees of vertices in some graph, it is called {\it graphical}. Such a sequence, when arranged in non-increasing order with $d_1 \geq d_2 \geq \cdots \geq d_n$, is referred to as a {\it degree sequence}.
The characterization of extremal graphs, concerning $\sigma$-irregularity, with a given degree sequence was recently undertaken in \cite{dglc-etfdssi-2023}.
 
To establish that a given sequence of non-negative integers represents the degree sequence of some graph, one can utilize the following characterization by Erd\H{o}s and Gallai \cite{eg-ggdv-60}.

\begin{theorem}\label{te-erdos-gallai-modified}
For any integer $n \geq 1$ and non-increasing sequence $d_1 \geq d_2 \geq \dots \geq d_n$, there exists a graph with $n$ vertices having the respective degrees $d_1, d_2, \dots, d_n$ if and only if two conditions are satisfied:
\begin{enumerate}
\item $\sum_{i=1}^n d_i$ is even;
\item For all $1 \leq k \leq n-1$, the inequality $d_1 + d_2 +\dots+ d_k \leq k(k-1)+\sum\limits_{i=k+1}^{n} \min(k,d_i)$ holds.
\end{enumerate}
\end{theorem}%

A connected graph with maximum degree at most $4$ is commonly referred to as a {\it chemical graph}, while a non-cyclic chemical graph is denoted as a {\it chemical tree}. In \cite{adg-gmsi-2018}, it was established that among the chemical trees explored, the path graph exhibits the smallest $\sigma$-irregularity. Additionally, \cite{kpvsd-sict-2024} provided a characterization of chemical trees with maximal $\sigma$-irregularity.

The graphs with the same degree sequence do not necessarily have the same $\sigma$-irregularity 
(see Figure~\ref{fig-different} from \cite{ds-siiip-2023} for an example).
\begin{figure}[htb]
\begin{center}
\vspace{-0.1cm}
\includegraphics{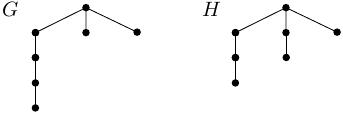}   
\caption{Illustration of two distinct graphs, $G_1$ and $G_2$, which share the degree sequence $1, 1, 1, 2, 2, 2, 3$. 
Despite their differing irregularity values ($\irr(G_1)=10$ and $\irr(G_2)=8$), they exhibit identical total irregularity 
($\irr_t(G_1)=\irr_t(G_2)=22$).}

\label{fig-different}
\vspace{-0.1cm}
\end{center}
\end{figure}
To avoid this, a variant of ${\sigma}$-irregularity, which is invariant with respect to a given degree sequence, 
was introduced in \cite{ds-siiip-2023}. It is called the {\it total ${\sigma}$-irregularity} and is defined as
\beq \label{def:sigma_t}
\sigma_t(G) =  \sum_{\{u,v\} \subseteq V(G)} \left(d_G(u)-d_G(v)\right)^2. \nonumber
\eeq

The publication \cite{ds-siiip-2023} provided initial insights into $\sigma_{t}$ by relating it to the first Zagreb index, defined as
$
M_1(G) = \sum_{v \in V} d(v)^2.
$
It established the formula $\sigma_{t}(G) = n M_{1}(G) - 4m^{2}$ for simple connected graphs, offering a quantitative understanding of $\sigma_{t}$ within this context.
Furthermore,  \cite{ds-siiip-2023} offered insights into the behavior of $\sigma_{t}$ in tree structures, 
revealing that the star tree possesses the maximum $\sigma_{t}$, while the path graph has the minimum $\sigma_{t}$.

In a recent study \cite{fdks-srosti-2024}, it was shown that $\sigma_{t}$ is equal to the degree variance of the graph. 
Leveraging this finding, the investigation proceeded to characterize irregular graphs and irregular bipartite graphs 
possessing maximal $\sigma_{t}$-irregularity. 
Additionally, the same study \cite{fdks-srosti-2024} provides various upper and lower bounds for the $\sigma_{t}$-irregularity index. 
Through the application of Fiedler's characterization of the largest and second smallest Laplacian eigenvalues of the graph, 
the authors in \cite{fdks-srosti-2024} established new relationships between $\sigma_{t}$ and $\sigma$.
These results deepen our understanding of $\sigma_{t}$ behavior also across different graph structures.

A {\em clique} within a graph $G$ refers to a subgraph in which every pair of vertices is connected by an edge. 
The {\em union} of two graphs $G_1$ and $G_2$, denoted as $G = G_1 \cup G_2$, is constructed by combining 
the disjoint vertex sets $V_1$ and $V_2$ and their corresponding edge sets $E_1$ and 
$E_2$ into $V = V_1 \cup V_2$ and $E = E_1 \cup E_2$. The {\em join} operation on graphs $G_1$ 
and $G_2$, denoted $G = G_1 + G_2$, involves taking the union $G = G_1 \cup G_2$ and 
adding edges between each vertex in $V_1$ and every vertex in $V_2$.
If a graph can be partitioned into a clique and an independent set, it is known as a {\em split graph}.
In the context of $\sigma_{t}$-irregularity, as well as in the case of the irregularity ${\rm irr}$, bidegreed graphs, specifically split graphs, are known to maximize these indices.
Furthermore, some graphs exhibit high values of ${\rm irr}$ and $\sigma_{t}$-irregularity while having minimal degree sets, a trait not typically associated with highly irregular graphs.
These findings are somewhat surprising, as one would typically expect that maximum irregularity measures are achieved by graphs with as many distinct degrees as possible.

Our aim is to generalize the $\sigma_t$ irregularity index so that its minimum is still attained by regular graphs, 
but its maximum will be attained by the antiregular graph. Specifically, we define the index $\IR(G)$ as follows:
$$
\IR(G)= \sum_{\{u,v\} \subseteq V(G)}|d_G(u)-d_G(v)|^{f(n)},
$$
where $n=|V(G)|$ and $f(n)$ is a function defined for $n\ge 4$. We note that the cases when $n\le 3$ are trivial.
The function $f(n)$ can be arbitrary, even constant, but it is expected that larger differences in degrees yield a greater contribution. Therefore, we require $f(n)>0$ to ensure meaningful results.
The value of $\IR(G)$ is $0$ if and only if $G$ is regular; otherwise, $\IR(G)>0$. However, our goal is to ensure that $\IR(G)$ is maximized when $G$ is antiregular. This poses a challenge when $f(n)$ is constant because for large $n$, one significant difference in degrees may outweigh numerous smaller differences (such as $1$, for example).
To address this issue, we assume that $\lim_{n\to\infty}f(n)=0$. While this requirement is not necessarily sufficient, it serves as a starting point for our exploration.

\section{Antiregular extension of $\sigma_t$-irregularity }

\subsection{General graphs}

\begin{theorem}
\label{thm:gr_bin}
Let $0 < f(n)\le\log_{n-2}\big(\frac{n^2-n-2}{n^2-n-4}\big)$ and let $n\ge 4$.
Then $\IR(G)$ achieves its maximal value if and only if $G$ is antiregular.
\end{theorem}

\begin{proof}
Consider $G$ as an antiregular graph, where vertices $u, v \in V(G)$ satisfy $d_G(u) \ne d_G(v)$.
Since $f(n)>0$, we have $|d_G(u)-d_G(v)|^{f(n)}\ge 1$. Taking into account that 
certain differences in degrees exceed 1 (recall that $n\ge 4$), we can deduce that
\begin{align}
\IR(G) > \tbinom n2-1=\frac{n^2-n-2}2.  \nonumber
\end{align}

On the other hand, the difference between the degrees of any two vertices is at most $n-2$. 
If $H$ is not antiregular, then it holds that:
\begin{align}
\IR(H)\le\big(\tbinom n2-2\big)(n-2)^{f(n)}=\frac{n^2-n-4}2(n-2)^{f(n)}. \nonumber
\end{align}

Starting with $f(n) \leq \log_{n-2}\left(\frac{n^2-n-2}{n^2-n-4}\right)$, we can derive the inequality
\[
\frac{n^2-n-2}{2} \geq \frac{n^2-n-4}{2} (n-2)^{f(n)},
\]
which, together with the inequalities for $\IR(G)$ and $\IR(H)$ above, leads to $\IR(G) > \IR(H)$.
\end{proof}

If $S$ is a sequence of $n$ integers, denote
$$
\IR(S)=\sum_{a,b\in S}|a-b|^{f(n)}.
$$
In the proof of Theorem~{\ref{thm:gr_bin}} we did not use the fact that the sequence of degrees is obtained from a real graph.
So we proved the following.

\begin{corollary}
\label{cor:se:bin}
Let $f(n)=\log_{n-2}\big(\binom n2-1\big)/\big(\binom n2-2\big)$, $n\ge 4$, and let $S=\{a_i\}_{i=1}^n$ be a sequence of integers such that $1\le a_i\le n-1$ for every $i$, $1\le i\le n$.
Then $\IR(S)$ is maximal if and only if $\{a_1,a_2,\dots,a_n\}=\{1,2,\dots,n-1\}$.
\end{corollary}

However, for general sequences we can prove a result stronger than Theorem~{\ref{thm:gr_bin}}.
Observe that if $n\ge 4$, then $\log_{n-2}\big((x+1)/x\big)$ is a decreasing function 
for $x\ge 1$, and therefore, $\log_{n-2}\frac{n-1}{n-2}>\log_{n-2}\big(\binom n2-1\big)/\big(\binom n2-2\big)$.

\begin{theorem}
\label{thm:se_lin}
Let $f(n)=\log_{n-2}\frac{n-1}{n-2}$, $n\ge 4$, and let $S=\{a_i\}_{i=1}^n$ be a sequence of integers such that $1\le a_i\le n-1$ for every $i$, $1\le i\le n$.
Then $\IR(S)$ is maximal if and only if $\{a_1,a_2,\dots,a_n\}=\{1,2,\dots,n-1\}$ or if $n=4$ and $S=(1,1,3,3)$.
\end{theorem}

\begin{proof}
By way of contradiction, suppose that $S$ achieves the maximum value for $\IR$, but some integer value from $[1,n-1]$ is missing in $S$, say $b$.
Let $c$ be a value which occurs most often in $S$.
Obviously, $c$ occurs in $S$ at least twice.
Now remove one occurrence of $c$ in $S$, replace it by $b$, and denote the resulting sequence by $S'$.
We show that $\IR(S')>\IR(S)$.

Suppose that $S$ contains exactly $k$ values distinct from $c$.
Obviously $k\ge 1$, since if $k=0$ then $\IR(S)=0$.
Also $k \leq n - 3$, since $c$ occurs at least twice.
Analogously to the proof of Theorem~{\ref{thm:gr_bin}}, the contribution of one occurrence $c$ 
to $\IR(S)$ is at most $(n-2)^{f(n)}\cdot k$, while the contribution of the element with value $b$ to $\IR(S')$ is at least $k+1$.
Therefore, 
$\IR(S) - (n-2)^{f(n)}\cdot k + k+1 \leq \IR(S')$.
Since $k \leq n - 3$, it follows that $-(n-2)^{f(n)} \cdot k + k + 1 > 0$ for $f(n) = \log_{n-2} \frac{n-1}{n-2}$. 
Consequently, we obtain that $\IR(S') > \IR(S)$.

If $\IR(S')=\IR(S)$, then $k=n-2$, so $c$ occurs exactly twice in $S$, 
and $|c-a_i|=n-2$ for all $a_i$ in $S$, where $a_i\ne c$. As a result, $S$ contains only values $1$ and $n-1$. 
Given that $n\ge 4$ and $c$ occurs most frequently in $S$, we deduce that $n=4$ and $S=(1,1,3,3)$. 
Consequently, $S'=(1,1,2,3)$ or $S'=(1,2,3,3)$. Therefore, $f(n)=\log_2(3/2)$, $\IR(S)=4\cdot 2^{f(n)}=6$, and $\IR(S')=2\cdot 2^{f(n)}+3\cdot 1=6$ as well.
\end{proof}

It is worth noting that the previous proof does not apply to graphic sequences, as $S'$ is not necessarily graphic if $S$ is.

For $n\ge 6$, it holds that $\frac{n-1}{n-2} < (n-2)^{1/n}$, implying $\log_{n-2}\frac{n-1}{n-2} < \frac{1}{n}$. 
By computer search, we determined that the antiregular graph attains the maximum value of $\IR(G)$ even when $f(n)=\frac{1}{n}$ and $n\le 11$.
Consequently, we present the following problem statement.

\begin{problem}
\label{prob:1/n}
Let $f(n)=\frac 1n$.
Is it true that the maximum value of $\IR(G)$ is achieved when $G$ is an antiregular graph?
\end{problem}

Observe that $\lim_{n\to\infty}(n-2)^{f(n)}=1$ if $f(n)=\log_{n-2}\frac{n-1}{n-2}$, while $\lim_{n\to\infty}(n-2)^{f(n)}=\infty$ if $f(n)=\frac 1n$.
Thus, if the answer to Problem~{\ref{prob:1/n}} is negative, does the negative result still hold when $f(n)=\frac{1}{n}$ is replaced by $f(n)=\log_{n-2}(c)$ 
for a constant $c>1$? Conversely, if the answer to Problem~{\ref{prob:1/n}} is positive, then the following problem arises.

\begin{problem}
\label{prob:c}
Let $f(n)=c$, where $c$ is a real number in the interval $(0,1)$.
Is it true that the maximum value of $\IR(G)$ is achieved when $G$ is an antiregular graph?
\end{problem}

\subsection{Trees}

Now we focus on trees.
The path and the star on $n$ vertices are denoted by $P_n$ and $S_n$, respectively.
For the minimal value of $\IR$, we have the following statement.

\begin{theorem}
\label{thm:min_tree}
Let $T$ be a tree with the minimum value of $\IR$.
\begin{enumerate}
\item If $f(n) > \log_{n-2}\left(\frac{2n-4}{n-1}\right)$, then $T  \cong P_n$.
\item If $f(n) = \log_{n-2}\left(\frac{2n-4}{n-1}\right)$, then $T  \cong P_n$ or $T  \cong S_n$.
\item If $f(n) < \log_{n-2}\left(\frac{2n-4}{n-1}\right)$, then $T  \cong S_n$.
\end{enumerate}
Moreover, if $f(n) < \log_{n-2}\left(\frac{2n-4}{n-1}\right)$, then $P_n$ is the tree with the second smallest value of $\IR$.
\end{theorem}

\begin{proof}
We begin with the path $P_n$. It has two vertices of degree $1$ and $n-2$ vertices of degree $2$. 
Therefore, $P_n$ contains $2(n-2)$ pairs of vertices whose degrees differ by $1$, 
while the other pairs of vertices contribute $0$ to $\IR$. Hence, we have:
$$
\IR(P_n) = 2n - 4.
$$

Let $T$ be a tree on $n$ vertices with $\IR(T)\le\IR(P_n)$.
Then $T$ cannot contain more than $2n-4$ pairs of vertices which degrees are different and if it contains exactly $2n-4$ such pairs then their degrees must differ by $1$.

Every tree contains at least two vertices of degree $1$.
If $T$ contains at least $2$ vertices of degree at least $2$, then $T$ has at least $2(n-2)$ pairs of vertices which degrees are different.
If this difference is always $1$ then $T$ contains only vertices of degrees $1$ and $2$, so $T$ is a path.
Thus, $T$ contains only one vertex of degree at least $2$, and hence $T$ is the star $S_n$.
As a consequence, one of $P_n$ and $S_n$ attains the minimum value of $\IR$.
And in the case when $S_n$ attains the minimum value of $\IR$, the path $P_n$ attains the second minimum value of it.

We have
\[
\IR(S_n) = (n - 1)(n - 2)^{f(n)},
\]
thus $\IR(P_n) = \IR(S_n)$ if and only if $f(n) = \log_{n - 2}\left(\frac{2n - 4}{n - 1}\right)$.
Furthermore, when $f(n) > \log_{n - 2}\left(\frac{2n - 4}{n - 1}\right)$, we have $\IR(P_n) < \IR(S_n)$, 
and when $f(n) < \log_{n - 2}\left(\frac{2n - 4}{n - 1}\right)$, $\IR(P_n) > \IR(S_n)$. 
\end{proof}

For many topological indices, $P_n$ and $S_n$ represent opposite extremes. Hence, it may come as a surprise that when $f(n) = \log_{n-2}\left(\frac{2n-4}{n-1}\right)$, both $P_n$ and $S_n$ achieve the minimum value of $\IR$, while all other trees have a larger $\IR$.

While $P_n$ achieves the second minimum value of $\IR$ if $f(n) < \log_{n-2}\left(\frac{2n-4}{n-1}\right)$, the star $S_n$ does not necessarily achieve the second minimum value of $\IR$ when $f(n) > \log_{n-2}\left(\frac{2n-4}{n-1}\right)$. For instance, if $f(n) = \frac{1}{n}$ and $n \geq 5$, then $\IR(Y_n) < \IR(S_n)$, where $Y_n$ is obtained from the claw $S_4$ by subdividing one of the edges exactly $n-4$ times (observe that $\IR(Y_n) = (4n - 16) + 3 \cdot 2^{f(n)}$).

The problem of finding trees with the maximum value of $\IR$ remains open.

\begin{problem}
\label{prob:max_tree}
Let $f(n)$ be a positive function with $\lim_{n\to\infty}f(n)=0$.
Find trees which attain the maximum value of $\IR$.
\end{problem}

\subsection{Chemical graphs}

A graph is considered chemical if its vertices have degrees at most $4$. If a chemical graph has $a_i$ vertices of degree $i$, $1\le i\le 4$, then its degree sequence is denoted by $(1^{a_1},2^{a_2},3^{a_3},4^{a_4})$.

It is evident that the minimum value of $\IR$ is attained by regular graphs, exemplified by structures like a 
cycle $C_n$ or graphs with degree sequences such as $(1^0,2^0,3^0,4^n)$. 
Henceforth, our focus lies on chemical graphs exhibiting the maximum value of $\IR$.

\begin{theorem}
\label{thm:chem_max}
Let $n\ge 7$, $f(n)\le\log_3\left(\frac{3n^2}{3n^2-8}\right)$, and let $(1^{a_1},2^{a_2},3^{a_3},4^{a_4})$ 
be a degree sequence of a chemical graph $G$ with the maximum value of $\IR(G)$.
Then,
\begin{itemize}
\item[1.] If $n=4k-1$, then $a_1=a_3=a_4=k$ and $a_2=k-1$;
\item[2.] If $n=4k$, then $a_1=a_2=a_3=a_4=k$;
\item[3.] If $n=4k+1$, then $a_1=a_2=a_3=k$ and $a_4=k+1$;
\item[4.] If $n=4k+2$, then either $a_1=a_3=k$ and $a_2=a_4=k+1$, or $a_1=a_3=k+1$ and $a_2=a_4=k$.
\end{itemize}
\end{theorem}

\begin{proof}
Let $S=(1^{a_1},2^{a_2},3^{a_3},4^{a_4})$ be a degree sequence of a chemical graph.
First, we show that if $a_i-a_j\ge 2$ for $1\le i,j\le 4$, then degree sequence $S'=(1^{a'_1},2^{a'_2},3^{a'_3},4^{a'_4})$, where $a'_t=a_t$ for $t\in\{1,2,3,4\}\setminus\{i,j\}$, $a'_i=a_i-1$ and $a'_j=a_j+1$, has more pairs of vertices of distinct degree.
Choose $p$ and $q$ so that $\{1,2,3,4\}=\{i,j,p,q\}$.
Then $S$ has $(a_i+a_j)(a_p+a_q)+a_pa_q+a_ia_j$ pairs of vertices of distinct degree, while $S'$ has $(a'_i+a'_j)(a_p+a_q)+a_pa_q+a'_ia'_j$ such pairs.
Consequently, $S'$ has $(a_i-1)(a_j+1)-a_ia_j=a_i-a_j-1$ pairs of vertices of distinct degree more than $S$.
Since $a_i-a_j\ge 2$, we have $a_i-a_j-1>0$.

Hence, a chemical graph with the largest number of pairs of vertices of distinct degree has a degree sequence $(1^{a_1},2^{a_2},3^{a_3},4^{a_4})$, where $|a_i-a_j|\le 1$ for all $i,j$ with $1\le i,j\le 4$. We denote the number of such pairs in a graph $G$ by $b$. Obviously, $b\le 6(\frac{n}{4})^2=\frac{3}{8} n^2$, and $\IR(G)\ge b\cdot 1$. If $H$ is a graph with a smaller number of pairs of vertices of distinct degree, then $\IR(H)\le (b-1)3^{f(n)}$. Hence, $\IR(G)\ge\IR(H)$ if $f(n)\le\log_3\left(\frac{b}{b-1}\right)$. Since $\frac{(\frac{3}{8} n^2)}{(\frac{3}{8} n^2-1)}\le\frac{b}{b-1}$, we have $\IR(G)\ge\IR(H)$ if $f(n)\le\log_3\left(\frac{3n^2}{3n^2-8}\right)$. In fact, we have $\IR(G)>\IR(H)$ in this case, since if $H$ has $b-1$ pairs of vertices of distinct degree, then it has a vertex of degree $2$ or $3$ (recall that $n\ge 7$), and so some differences are smaller than $3$ in $H$.

It remains to determine which degree sequence $(1^{a_1},2^{a_2},3^{a_3},4^{a_4})$, with $|a_i-a_j|\le 1$ for all $1\le i,j\le 4$, yields the largest value of $\IR$. We distinguish four cases.

{\bf Case 1}.
Let $n=4k$.
Let $n=4k$. Then the only sequence of the required type is $(1^k,2^k,3^k,4^k)$. 
By Theorem~\ref{te-erdos-gallai-modified}, this sequence is a degree sequence if $k\ge 2$.

\textbf{Case 2}.
Let $n=4k+1$. Then there are two possible sequences, namely $S_1=(1^k,2^{k+1},3^k,4^k)$ and $S_2=(1^k,2^k,3^k,4^{k+1})$, since the other two sequences yield an odd number of vertices of odd degree, which is impossible. Both $S_1$ and $S_2$ yield the same number of pairs of vertices with distinct degree. Also, the differences are the same, up to the differences between the vertices of degree $1$ and the extra vertex. The sequence $S_1$ has $k$ differences of $1$ which are missing in $S_2$, while $S_2$ has $k$ differences of $3$ which are missing in $S_1$. Since larger differences yield a larger value of $\IR$, the maximum value of $\IR$ is obtained by $S_2$. If $k\ge 2$, then 
by Theorem~\ref{te-erdos-gallai-modified}, $(1^k,2^k,3^k,4^{k+1})$ is a degree sequence.

{\bf Case 3}.
Let $n=4k-1$. The only possible sequences are $(1^k,2^{k-1},3^k,4^k)$ and $(1^k,2^k,3^k,4^{k-1})$. 
The sequence with more values $4$ yields a larger value of $\IR$. By Theorem~\ref{te-erdos-gallai-modified}, 
$(1^k,2^{k-1},3^k,4^k)$ is a degree sequence if $k\ge 2$.

{\bf Case 4}.
Let $n=4k+2$. The only possible sequences are $(1^k,2^{k+1},3^k,4^{k+1})$ and $(1^{k+1},2^k,3^{k+1},4^k)$ 
since the number of vertices of odd degree must be even. Both of these sequences yield the same differences, 
and according to Theorem~\ref{te-erdos-gallai-modified}, they are both degree sequences if $k\ge 2$.
 Consequently, they both yield graphs with the maximum value of $\IR$.
\end{proof}

Observe that if $G$ is the (connected) antiregular graph, then the multiplicity of degrees is only $1$ and $2$.
Hence, difference in multiplicities is at most $1$ in $G$.
From this point of view, graphs with degree sequences as in Theorem~{\ref{thm:chem_max}} can be regarded as antiregular chemical graphs.

Since for chemical graphs the difference between degrees of vertices is bounded, we conjecture the following.

\begin{conjecture}
\label{conj:chemical1/n}
The same graphs as in Theorem~{\ref{thm:chem_max}} have maximum value if $\IR$ even if $f(n)=\frac 1n$.
\end{conjecture}

\begin{conjecture}
The same graphs as in Theorem~{\ref{thm:chem_max}} have maximum value if $\IR$ even if $f(n)$ is a constant in the interval $(0,1)$.
\end{conjecture}

To support Conjecture~{\ref{conj:chemical1/n}, denote by $x_i$ the
number of vertices of degree $i$ in a chemical graph and consider the
case $n\equiv 0\pmod 4$.
Then the task is to maximize $(x_1x_2+x_2x_3+x_3x_4)\cdot
1^{1/n}+(x_1x_3+x_2x_4)\cdot 2^{1/n}+(x_1x_4)\cdot 3^{1/n}$ subject to
constraint $x_1+x_2+x_3+x_4=n$.
In real numbers the solution is $x_1=x_4=\frac n2\cdot\frac{\root n\of
2}{1+2\root n\of 2-\root n\of 3}$, $x_2=x_3=\frac n2\cdot\frac{1+\root
n\of 2-\root n\of 3}{1+2\root n\of 2-\root n\of 3}$, and
$\lim_{n\to\infty}(x_1-n/4)=\frac 18 \ln{3}\doteq 0.1373$.
Unfortunately, this does not directly imply that $x_1=x_2=x_3=x_4=\frac
n4$ is the integer solution.

\section*{Acknowledgements}
Martin Knor acknowledges partial support from Slovak research grants VEGA 1/0567/22, VEGA 1/0069/23, APVV-22-0005, and APVV-23-0076, as well as from the Slovenian Research Agency (ARIS) program P1-0383 and project J1-3002.
R. Škrekovski has been partially supported by ARIS program P1-0383, project J1-3002, and the annual work program of Rudolfovo.
S. Filipovski is partially supported by ARIS research program P1-0285 and research projects N1-0210, J1-3001, J1-3002, J1-3003, and J1-4414.
D. Dimitrov acknowledges partial support from ARIS program P1-0383 and projects J1-3002 and BI-US-24-26-073.


\begin{thebibliography}{99}

\bibitem{Dimit-Abdo}
H.~Abdo, S.~Brandt, D.~Dimitrov, \textit{The total irregularity of a graph},
Discrete Math. Theor. Comput. Sci. \textbf{16} (2014) 201--206.

\bibitem{adg-gmsi-2018}
H.~Abdo, D.~Dimitrov, I.~Gutman, \textit{ Graphs with maximal $\sigma$-irregularity},
Discrete Appl. Math.  \textbf{250} (2018) 57--64.

\bibitem{a-sag-2020}
A.~Ali, \textit{ A survey of antiregular graphs},
Contrib. Math.  \textbf{1} (2020) 67--79.

\bibitem{ar-timphda-2020}
A.~Ali, T.~R{\' e}ti, \textit{ Two irregularity measures possessing high discriminatory ability},
Contrib. Math.  \textbf{1} (2020) 27--34.

\bibitem{aaabh-msikcgi-2023}
A.~Ali, A.~M.~Albalahi, A.~M.~Alanazi, A.~A~Bhatti, A.~E.~Hamza, \textit{ On the maximum sigma index of 
k-cyclic graphs},
Discrete Appl. Math.  \textbf{352} (2023) 58--62.




\bibitem{AlbertsonIrr} 
M.~O.~Albertson, \textit{The irregularity of a graph}, Ars Comb. \textbf{46} (1997) 219--225.
%

\bibitem{bc-ngp-1967}
M. Bezhad, G. Chartrand, \textit{No graph is perfect}, Amer. Math. Monthly  \textbf{74} (1967) 962--963.

\bibitem{bm-gt-2008} 
A.~Bondy,   U.~S.~R.~Murty, \textit{Graph Theory}, Springer Verlag, Berlin, 2008.



\bibitem{dglc-etfdssi-2023} 
D.~Dimitrov, W.~Gao, W.~Lin, J.~Chen, 
\textit{Extremal trees with fixed degree sequence for $\sigma$-irregularity},
Discrete Math. Lett. \textbf{12} (2023) 166--172.

\bibitem{Dar-Riste} 
D.~Dimitrov, R.~{\v S}krekovski, \textit{Comparing the irregularity and the total irregularity of graphs},  Ars Math. Contemp.
 \textbf{9} (2015) 25--30.

\bibitem{ds-siiip-2023}
D.~Dimitrov, D.~Stevanovi{\' c}, \textit{On the $\sigma_t$-irregularity  and  the inverse  irregularity problem},
Appl. Math. Comput. \textbf{441} (2023) 127709.

\bibitem{eg-ggdv-60}
P.~Erd\H{o}s, T.~Gallai, Graphs with prescribed degrees of vertices, 
(in Hungarian) \textit{Mat. Lapok.} \textbf{11} (1960) 264--274.

\bibitem{fdks-srosti-2024}
S.~Filipovski, D.~Dimitrov, M.~Knor, R.~{\v S}krekovski, \textit{Some results on $\sigma_t$-irregularity},
to appear in Ars Math. Contemp.

\bibitem{gtycc-ipsi}
I. Gutman, M. Togan, A. Yurttas, A. S. Cevik, I. N. Cangul, 
\textit{Inverse problem for sigma index}, MATCH Commun. Math. Comput. Chem. \textbf{79} (2018) 491--508.


\bibitem{kpvsd-sict-2024}
\v Z. Kovijani\'c Vuki{\'c}evi\' c, G. Popivoda, S. Vujo\v sevi\'c, R.
\v Skrekovski, D. Dimitrov, 
\textit{The $\sigma$-irregularity  of  chemical trees},
MATCH Commun. Math. Comput. Chem. \textbf{91} (2024) 267--282.



\bibitem{r-spgiiprsi-2019}
T.~R{\' e}ti, \textit{On some properties of graph irregularity indices with a particular regard to the $\sigma$-index},
Appl. Math. Comput. \textbf{344} (2019) 107--115.

\end{thebibliography}
\end{document}